\documentclass[12pt,a4paper,reqno]{amsart}
\usepackage{graphicx}
\usepackage[linktocpage=true,colorlinks,citecolor=blue,linkcolor=blue,urlcolor=blue]{hyperref}
\usepackage{amssymb}
\usepackage{cite}
\usepackage{bbm}
\usepackage{amsmath}
\usepackage{latexsym}
\usepackage{amscd}
\usepackage{tikz}
\usepackage{amsthm}
\usepackage{mathrsfs}
\usepackage{url}
\usepackage[utf8]{inputenc}
\usepackage[english]{babel}
\usepackage{amsfonts}
\usepackage{mathtools}
\usepackage{comment}
\usepackage[autostyle]{csquotes}
\usepackage[colorinlistoftodos]{todonotes}

\numberwithin{equation}{section}
\def\R{\mathbb R}
\def\Z{\mathbb Z}

\def\N{\mathbb N}
\def\E{\mathbb E}

\def\ee{\varepsilon}

\def\gcd{\operatorname{gcd}}

\DeclarePairedDelimiter\floor{\lfloor}{\rfloor}

\def\le {\leqslant}
\def\ge {\geqslant}

\newtheorem{theorem}{Theorem}[section]
\newtheorem{lemma}[theorem]{Lemma}
\newtheorem{proposition}[theorem]{Proposition}

\theoremstyle{definition}

\newtheorem*{remark}{Remark}
\newtheorem{question}[theorem]{Question}
\numberwithin{equation}{section}

\theoremstyle{remark}

\usepackage{bm,physics}
\newcommand{\defeq}{:=}
\newcommand{\maps}{\colon}

\newcommand{\mcal}{\mathcal}

\def\E{\mathbb E}
\def\Q{\mathbb Q}

\title{Average sizes of mixed character sums}
\author{Victor Y. Wang}
\address{IST Austria, Am Campus 1, 3400 Klosterneuburg, Austria}
\email{vywang@alum.mit.edu}
\author{Max Wenqiang Xu}
\address{Courant Institute, 251 Mercer Street, New York 10012, USA}
\email{maxxu1729@gmail.com}
\dedicatory{To Melvyn Nathanson and Carl Pomerance on the occasion of their eightieth birthdays}

\begin{document}

\begin{abstract}
We prove that the average size of a mixed character sum
$$\sum_{1\le n \le x}  \chi(n) e(n\theta) w(n/x)$$
(for a suitable smooth function $w$)
is on the order of $\sqrt{x}$ for all irrational real $\theta$ satisfying a weak Diophantine condition, where $\chi$ is drawn from the family of Dirichlet characters modulo a large prime $r$ and where $x\le r$.
In contrast, it was proved by Harper that the average size is $o(\sqrt{x})$ for rational $\theta$.
Certain quadratic Diophantine equations
play a key role in the present paper.
\end{abstract}

\maketitle

\section{Introduction}
Throughout this paper, let $r$ be a large prime.
In a recent breakthrough work of Harper\cite{harpertypical}, it is established that the typical size of Dirichlet character sum $\sum_{1\le n \le x} \chi(n)$ is $o(\sqrt{x})$. Precisely, it is proved that if $1\le x\le r$ with $\min(x,r/x)\to \infty$, then
\begin{equation}
\label{harper}
 \frac{1}{r-1}   \sum_{\chi \bmod r} |\sum_{1\le n \le x} \chi(n)| = o(\sqrt{x}),
\end{equation}
which would imply that for ``typical $\chi$", the character sum is of size $o(\sqrt{x})$.\footnote{In both the random model and the character sum cases, Harper proved a sharp quantitative saving.}
This somewhat surprising phenomenon was first proved in an earlier beautiful work of Harper \cite{HarperLow} in the random model case, namely instead of for a character $\chi$, the ``better than square-root cancellation" phenomenon holds for a random sum $\sum_{1\le n \le x} f(n)$, where $f(n)$ is a random multiplicative function,\footnote{$f(p)$ is defined as an independent identically distributed random variable taking values on complex unit circle for all primes $p$, and $f(n)$ is defined as a completely multiplicative function.}
which can be viewed as a random model for Dirichlet characters. In order to connect the deterministic case to the random multiplicative function model, Harper used a derandomization method.  

The underlying reason for the ``better than square-root cancellation" happening is subtle, which connects to the study of critical multiplicative chaos in probability theory. It is also believed that such a structure is delicate and may easily be destroyed by some perturbation. For example, one can consider general character sums like $\sum_{1\le n \le x} a(n) \chi(n)$ (see \cite{harpertypical, WXratio}) or random sums $\sum_{1\le n \le x} a(n)f(n)$ (see \cite{SoundXu, GWCLT, KSX}) with a twist $a(n)$. It is believed that for some generic coefficients $a(n)$, such a ``better than square-root cancellation" would die out (see \cite[\S 1]{Xu} for more discussions). 

In this paper, we are interested in \emph{mixed} character sums of the form $\sum_{1\le n \le x} \chi(n) e(n\theta)$, featuring both an additive character $e(n\theta)$ and a multiplicative character $\chi(n)$. 

For technical reasons it is simpler for us to work with a smooth weight $w\maps\R\to \R$.
Throughout this paper,
fix a smooth function $w\maps \R\to \R$,
supported on $[0,1]$,
with $\int_0^1 w(t)^2\, dt > 0$.
Note that by partial summation, \eqref{harper} immediately implies
\begin{equation}
 \frac{1}{r-1}   \sum_{\chi \bmod r} |\sum_{n\ge 1}  \chi(n) w(n/x)| = o(\sqrt{x}).
\end{equation}
Similarly,
it is also known that for any fixed \emph{rational} number $\theta$ we have
\begin{equation}
 \frac{1}{r-1}   \sum_{\chi \bmod r} |\sum_{n\ge 1} \chi(n) e(n\theta) w(n/x) | = o(\sqrt{x}).
\end{equation}
Our main theorem shows that the typical size will be on the square-root size when $\theta$ is not ``too close to the rationals".
(The precise condition we impose on $\theta$ is \eqref{1.4},
which is satisfied for most irrational numbers, including $\pi$, $e$, and any algebraic irrational $\theta$.)
\begin{theorem}\label{THM: main}
Let $\theta$ denote an irrational real number such that for some positive constant $C=C(\theta)$ and all $q \in {\mathbb N}$, the following holds:
\begin{equation} \label{1.4}
\Vert q \theta \Vert  := \min_{n\in {\mathbb Z}} |q\theta -n | \ge C \exp( - q^{\frac{1}{4}}).
\end{equation}
If $x$ is sufficiently large (in terms of $w$), then we have for all $ x\le r$, 
\begin{equation}\label{main-thm-goal}
\sqrt{x}
\ll \frac{1}{r-1}\sum_{\chi \bmod r} | \sum_{1\le n \le x}  \chi(n) e(n\theta) w(n/x)|
\ll \sqrt{x},
\end{equation}
where the implied constants are allowed to depend on $\theta$ (and $w$).
\end{theorem}

Henceforth, to simplify the notation
we let $\E_{\chi} S(\chi) \defeq \frac{1}{r-1}\sum_{\chi \bmod r} S(\chi)$.

We give an overview of the proof. The upper bound on the first absolute moment follows easily from the second moment computation together with the Cauchy--Schwarz inequality. Indeed, 
for $1\le x\le r$, the orthogonality relation implies that
\begin{equation}\label{eqn: 2nd}
 \E_\chi [|\sum_{1\le n \le x } \chi(n) e(n\theta) w(n/x) |^2]
= \sum_{1\le n\le \min(x,r-1)} w(n/x)^2
\sim x \int_0^1 w(t)^2\, dt \asymp x,   
\end{equation}
provided that $x$ is sufficiently large (in terms of $w$).

A key feature of our result is that the average size is around $\sqrt{x}$ uniformly \textit{for all} $x\le r$.
However, our methods diverge around $x\asymp \sqrt{r}$. 
In the case $x\le \sqrt{r}$ we show the bound without the weight, that is, for $x\le \sqrt{r}$, 
 \begin{equation}\label{eqn: 4 upper}
    \E_\chi[|\sum_{1\le n \le x } \chi(n) e(n\theta)|^{4}] \ll x^{2}.
\end{equation}
By using partial summation, this would imply 
\begin{equation}\label{eqn: 4 upper weight}
    \E_\chi[|\sum_{1\le n \le x } \chi(n) e(n\theta) w(n/x)|^{4}] \ll x^{2}.
\end{equation}
An application of  H\"{o}lder's inequality together with \eqref{eqn: 2nd} and \eqref{eqn: 4 upper weight} establishes the lower bound on the first absolute moment.
By using the orthogonality, the fourth moment of the mixed character sum in \eqref{eqn: 4 upper} is the same as 
\begin{equation}\label{eqn: cong}
  \sum_{\substack{m_1m_2 \equiv n_1n_2 \bmod r \\ m_i, n_i \le x}} e((m_1+m_2- n_1-n_2)\theta).   
\end{equation}
The congruence condition is irrelevant if $x$ is small enough compared with $r$. In the above equation~\eqref{eqn: cong} when $x\le \sqrt{r}$, the ``$\equiv$" is the same as the ``$=$". One may view the $\chi(\cdot)$ as a random variable distributed on the unit circle. Indeed, the fourth moment above is the same as $ \E_{f} [ | \sum_{1\le n \le x} f(n) e(n\theta)|^{4}] $ where $f$ is a Steinhaus random multiplicative function. The parallel result for the random multiplicative function is known. It is proved that \cite{SoundXu} for the family of $\theta$ satisfying a certain weak Diophantine condition (but stronger than our requirement in Theorem~\ref{THM: main}),
the quantity $\sum_{1\le n \le x} f(n) e(n\theta)$ behaves Gaussian with variance $x$,
which as a consequence, shows that its typical size is $\sqrt{x}$. An earlier result \cite{BNR} proved this for almost all $\theta$ by using the method of moments. We will employ the methods in \cite{SoundXu} to deduce our result in the $x\le \sqrt{r}$ range in Section~\ref{Sec: short}.

The bulk of the paper focuses on the case $x\ge \sqrt{r}$. In this case, we do not directly compute the fourth moment. Instead, we do a Poisson summation first to study a dual problem, with $r/x$ replacing $x$.
This completes the sum over $n$
modulo $r$, leading to a new problem involving Diophantine approximation and congruences.
Recently, Heap and Sahay \cite{heap2025fourth}
also analyzed Diophantine approximations, in a different way,
when bounding the fourth moment of the Hurwitz zeta function.
Their work is closely tied to partial sums of the form $\sum_{1\le n\le x} n^{it} e(n\theta)$, featuring $n^{it}$ rather than $\chi(n)$.

\begin{remark}
A restriction like $x\le r$ is necessary in Theorem~\ref{THM: main}.
For instance, if $x\equiv 0\bmod{r}$, then
\begin{equation*}
\E_{\chi} \abs{\sum_{1\le n \le x} \chi(n) e(n\theta)}
= \frac{\abs{1-e((x-r)\theta)}}{\abs{1-e(r\theta)}}
\E_{\chi} \abs{\sum_{1\le n \le r} \chi(n) e(n\theta)}
\ll \frac{1}{\abs{1-e(r\theta)}} \sqrt{r},
\end{equation*}
by the easy ($\ll$) direction of \eqref{main-thm-goal}.
But if $x/r\to \infty$, then for infinitely many primes $r$ the right-hand side is $\ll \sqrt{r}$, not $\asymp \sqrt{x}$.

For instance, if $c\ge 2$ such that $p_n+c$ is prime for infinitely many primes $p_n$,
then for some $r_n \in \{p_n,p_n+c\}$, we have, by the triangle inequality
\begin{equation*}
\begin{split}
\abs{1-e(r_n\theta)}
&= \max(\abs{1-e(p_n\theta)}, \abs{1-e((p_n+c)\theta)}) \\
&\gg \abs{e(p_n\theta)-e((p_n+c)\theta)}
= \abs{1-e(c\theta)}
\gg_\theta 1.
\end{split}
\end{equation*}
\end{remark}

Our main result concerns the size of the mixed character sums for the large range where the two key parameters satisfy $x\le r$. It would be nice if one could determine the distribution. In particular, we raise the following question.
\begin{question}\label{ques}
For which ranges of parameters $x,r\to +\infty$
is it true that
as $\chi$ varies uniformly over the family of Dirichlet characters mod $r$, we have
  \[   \frac{1}{\sqrt{x}}\sum_{1\le n \le x} \chi(n) e(\sqrt{2}n) \xrightarrow{d} \mathcal{CN}(0,1)?\]
 \end{question} 
In the above notation, $\mathcal{CN}(0,1)$ denotes standard complex Gaussian distribution.
One may also ask about any $\theta$ that is not ``close to" rationals;
we put $\theta=\sqrt{2}$ for concreteness.
Originally, we conjectured that Question~\ref{ques} should have an affirmative answer for all $1\le x\le r$.
However, very recent work of Bober, Klurman, and Shala
\cite{bober2025distribution}
shows, in particular, that Question~\ref{ques} in fact has a negative answer for $x=r$.

To attack Question~\ref{ques}, starting with the method of moments is natural. For $x\le r^{\epsilon}$, we have perfect orthogonality even for high moments computation and the problem is essentially the same as replacing $\chi$ with a random multiplicative function. This might be doable; for example, we refer the readers to \cite{PWX, WXPaucity} for high moments computation on related problems. It would be even more interesting to determine the limiting distribution for the full range $x\le r$, for which the periodicity modulo $r$ can no longer be ignored.  \\

There has been a line of nice work on studying the character sums with additive twists without averaging over characters, and we refer readers to \cite{MV77, BG22} and references therein. We also mention that there is the dual direction of the problem, where one can average over $\theta$ but fix a multiplicative function, e.g.~the Liouville function $\lambda(n)$ (see \cite{PR1} for the most recent development and reference therein).

\subsection*{Acknowledgements}
We thank
Ofir Gorodetsky,
Andrew Granville,
Adam Harper,
Youness Lamzouri,
Kannan Soundararajan,
Ping Xi,
and Matt Young,
for their interest, helpful discussions and comments.
Special thanks are due to Jonathan Bober, Oleksiy Klurman, and Besfort Shala
for sending us a letter about Question~\ref{ques},
and to Hung Bui for informing us of \cite{heap2025fourth}.
V.Y.W. thanks Stanford University for its hospitality,
and is supported by the European Union's Horizon~2020 research and innovation program under the Marie Sk\l{}odowska-Curie Grant Agreement No.~101034413.
M.W.X. is supported by a Simons Junior Fellowship from the Simons Society of Fellows at Simons Foundation.
Finally, we thank the anonymous referee for their comments.

\section{Short sum case}\label{Sec: short}
In this section, we prove the lower bound in Theorem~\ref{THM: main} in the case $x\le \sqrt{r}$. As discussed in the introduction, it suffices to prove an upper bound as in \eqref{eqn: 4 upper}. 
Expanding the fourth moment and noticing the perfect orthogonality (thanks to $x\le \sqrt{r}$), it is sufficient to show the following proposition. 

\begin{proposition}~\label{prop: short sum}
Let $\theta$ satisfy the Diophantine condition \eqref{1.4}. Let $x$ be large. Then 
\begin{equation}\label{eqn: really}
\Big| \sum_{\substack{1\le  m_1, m_2, n_1, n_2 \le x \\ m_1 m_2= n_1 n_2 \\ \{m_1, n_1\} \neq \{m_2, n_2\} }} e((m_1+m_2- n_1-n_2)\theta) \Big| =o (x^{2}).
\end{equation}
\end{proposition}
The proof of the Proposition closely follows the proof of \cite[Theorem 1.6, Theorem 3.1]{SoundXu}.
Our situation is simpler and we give a proof here for completeness.
We remark that the main difference is that the exponential sum in \cite{SoundXu} has multiplicative coefficients, which requires a result of Montgomery--Vaughan \cite{MV2007}.
Instead, we use the following simple estimate which leads to the weaker Diophantine condition \eqref{1.4} and thus a better quantitative result.  

A simple starting point is that 
\begin{equation}\label{eqn: geo}
    \sum_{1\le n \le y} e(n\alpha) = \frac{e(\alpha)(1-e(y\alpha))}{1-e(\alpha)} \ll \min \{ y,  \frac{1}{\|\alpha\|}\}. 
\end{equation}
\begin{proof}[Proof of Proposition~\ref{prop: short sum}]
We use the same parametrization as in \cite{SoundXu},
to write $m_1=ga$, $m_2 =hb$, $n_1= gb$ and $n_2=ha$.  The constraints on $m_1$, $m_2$, $n_1$, $n_2$ then become $(a,b)=1$ with $a\neq b$,  $g\neq h$, and $\max(a,b) \times \max(g,h)\le x$.  
Thus the sum we wish to bound becomes 
\begin{equation} 
\label{8.2} 
\sum_{\substack{ \max (a,b ) \times \max(g, h) \le x \\ a\neq b, \ (a,b)=1 \\g\neq h}} e((g-h)(a-b) \theta). 
\end{equation} 
Since $\max(a,b) \times \max(g,h) \le x$, we may break the sum above into the cases (1) when $\max(g,h) \le \sqrt{x}$, (2) when $\max(a,b )\le \sqrt{x}$, taking 
care to subtract the terms satisfying (3) $\max(a,b)$ and $\max(g,h)$ both below $\sqrt{x}$.  

Before turning to these cases, we record a preliminary lemma which will be useful in our analysis. 

\begin{lemma} \label{lem8.2} Let $\theta$ be an irrational number satisfying the Diophantine condition \eqref{1.4}. 
Let ${\mathcal L} = {\mathcal L}(x)$ denote the set of all integers $\ell$ with $|\ell| \le \sqrt{x}$ such that for some 
$k \le (\log x)^{1+\epsilon}$ one has $\Vert k\ell \theta \Vert \le x^{-\frac 13}$, where $\epsilon>0$ is small.  Then $0$ is in ${\mathcal L}$, and for any two distinct elements $\ell_1$, $\ell_2 \in {\mathcal L}$ 
we have $|\ell_1 - \ell_2| \gg (\log x)^{3/2}$.
\end{lemma} 
\begin{proof} Evidently $0$ is in ${\mathcal L}$, and the main point is the spacing condition satisfied by elements of ${\mathcal L}$.  If $\ell_1$ and $\ell_2$ are distinct elements of ${\mathcal L}$ then there exist $k_1$, $k_2 \le (\log x)^{1+\epsilon}$ with $\Vert k_1 \ell_1 \theta\Vert \le x^{-\frac 13}$ and $\Vert k_2 \ell_2 \theta \Vert \le x^{-\frac 13}$.  It follows that 
$\Vert k_1 k_2 (\ell_1 -\ell_2) \theta\Vert \le 2(\log x)^{1+\epsilon} x^{-\frac 13}$.  The desired bound on $|\ell_1 -\ell_2|$ now follows from the Diophantine property that we required of $\theta$, namely that $\Vert q\theta \Vert \gg \exp(-q^{\frac 1{4}})$  which implies that $(\log x  )^{4} \ll k_1 k_2 |\ell_1-\ell_2|$ and the desired estimate follows.
\end{proof} 
\subsection*{Case 1: \texorpdfstring{$\max (g, h )\le \sqrt{x}$}{max(g,h) <= sqrt(x)}} Suppose that $g$ and $h$ are given with $g$ and $h$ below $\sqrt{x}$ and $g\neq h$, and consider the sum over $a$ and $b$ in \eqref{8.2}.  
We distinguish two sub-cases, depending on whether $g-h$ lies in ${\mathcal L}$ or not.  Consider first the situation when $g-h \not\in {\mathcal L}$.  Using M{\" o}bius inversion to detect the condition that $(a,b)=1$, the sums over $a$ and $b$ may be expressed as (the $O(1)$ error term accounts for the term $a=b=1$ which must be omitted)
\begin{align}\label{8.3}  
&\sum_{\substack{ k\le x/\max(g, h)\\ }} \mu(k) \sum_{\substack{ t, s \le x/(k\max (g,h)) \\  }} e(k(g-h)(t-s)\theta) +O(1) 
\nonumber \\
&= \sum_{\substack{ k\le x/\max(g, h)\\ }} \mu(k) \Big|\sum_{\substack{ t \le x/(k\max (g,h)) \\  }} e(k(g-h)t\theta)\Big|^2 +O(1). 
\end{align}
If $k> (\log x)^{1+\epsilon}$ then we bound the sum over $t$ above by $x/(k\max(g,h))$, and so these terms contribute to \eqref{8.3} an amount 
$$
\ll \sum_{k>(\log x)^{1+\epsilon/2}} \frac{x^2}{k^2 \max(g,h)^2} \ll \frac{x^2}{(\log x)^{1+\epsilon/2} \max (g, h)^2}.
$$
Now consider $k\le (\log x)^{1+\epsilon/2}$. Since $g-h \not \in {\mathcal L}$ by assumption, it follows that $\Vert k(g-h) \theta \Vert >x^{-1/3} $.
An application of \eqref{eqn: geo} shows that the sum over $t$ in \eqref{8.3} 
is $\le x^{1/3}$.  Thus the terms $k \le(\log x)^{1+\epsilon/2}$ contribute to \eqref{8.3} an amount bounded by 
$$ 
\sum_{k\le (\log x)^{1+\epsilon/2}}x^{2/3} = O(x^{2/3 + \epsilon}).
$$ 
Summing this over all $g$, $h \le \sqrt{x}$, we conclude that the contribution of terms with $\max(g,h)\le \sqrt{x}$ and  $g-h \not\in {\mathcal L}$ to \eqref{8.2} is 
$$
\ll \sum_{g, h \le \sqrt{x}} \frac{x^2}{(\log x)^{1+\epsilon/2} \max(g,h)^2} \ll \frac{x^2}{(\log x)^{\epsilon}}. 
$$ 

Now consider the contribution of the terms $\max(g,h)\le \sqrt{x}$ where $g-h$ lies in ${\mathcal L}$ with $g-h \neq 0$.  Bounding the sum over $a$ and $b$ trivially by $\le (x/\max(g,h))^2$, we see that the 
contribution of these terms is 
$$ 
\ll \sum_{\substack{g \neq h \le \sqrt{x} \\ g- h \in {\mathcal L}}} \frac{x^2}{\max(g,h)^2} \ll \sum_{g\le \sqrt{x}} \frac{x^2}{g^2} \sum_{\substack{ h< g\\ g-h \in {\mathcal L}}} 1 
\ll  \sum_{g\le \sqrt{x}} \frac{x^2}{g^2} \frac{g}{(\log x)^{3/2}} \ll \frac{x^2}{(\log x)^{1/2}}, 
$$ 
where we used Lemma \ref{lem8.2} to bound the sum over $h$.

We conclude that the contribution of terms with $\max(g,h)\le \sqrt{x}$ to \eqref{8.2} is $\ll x^2/(\log x)^{1/2}$, completing our discussion of this case.  
 
 \subsection*{Case 2: \texorpdfstring{$\max\{a, b\}\le \sqrt{x}$}{max(a,b) <= sqrt(x)}}  We first consider the contribution in the case that $g=h$. Then the contribution in this case is 
 \begin{equation}\label{eqn: g=h}
     \ll \sum_{a, b\le \sqrt{x}} \sum_{g \le x/\max\{a, b\}}1 \ll x^{3/2}. 
 \end{equation}
From now on, we allow $g=h$, and thus 
 we must bound 
 $$ 
 \sum_{\substack{ a\neq b \le \sqrt{x} \\ (a, b)=1}} \Big| \sum_{\substack{g\le x/\max(a,b) \\ } } e(g(a-b)\theta)\Big|^2.
 $$
 Again we distinguish the cases when $a-b\in{\mathcal L}$, and when $a-b \not \in {\mathcal L}$.  In the first case, we bound the sum over $g$ above trivially by $\le x/\max(a,b)$, 
 and thus these terms contribute (using Lemma \ref{lem8.2})
 $$ 
 \ll \sum_{a\le \sqrt{x}} \frac{x^2}{a^2} \sum_{\substack{ b <a \\ a-b\in {\mathcal L}}} 1 \ll \sum_{a\le \sqrt{x}} \frac{x^2}{a^2} \frac{a}{(\log x)^{3/2}} \ll \frac{x^2}{(\log x)^{1/2}}. 
 $$ 
 
 Now consider the case when $a-b\not\in {\mathcal L}$. Again, since $a-b\not \in {\mathcal L}$, it follows that $\Vert (a-b)\theta \Vert >x^{-1/3}$ and an application of \eqref{eqn: geo} gives
 $$ 
 \sum_{\substack{g\le x/\max(a,b) \\ }} e(g(a-b)\theta) \ll x^{1/3}.
 $$ 
 Thus the contribution of the terms with $a-b\not \in {\mathcal L}$ is 
 $$ 
 \ll \sum_{a, b\le \sqrt{x}} x^{2/3} \ll x^{5/3}. 
 $$ 
 
 Thus the contribution to \eqref{8.2} from the Case 2 terms is $\ll x^2/(\log x)^{1/2}$.  
 
 \subsection*{Case 3: \texorpdfstring{$\max(a,b)$ and $\max(g,h) \le \sqrt{x}$}{max(a,b) and max(g,h) <= sqrt(x)}}  Similarly, we first bound the contribution from the case $g=h$ as in \eqref{eqn: g=h}. From now on, we allow $g=h$, and thus 
 we must bound 
 $$ 
 \sum_{\substack{ a\neq b \le \sqrt{x} \\ (a, b)=1}} \Big| \sum_{\substack{g\le \sqrt{x} \\ } } e(g(a-b)\theta)\Big|^2,
 $$ 
 and our argument in Case 2 above furnishes the bound $\ll x^2/(\log x)^{1/2}$.  
\end{proof}
\section{Reduction to a counting problem}

 In this section,   we consider the case $\sqrt{r}\le x \le r$
 and reduce \eqref{eqn: 4 upper weight} to
 a counting problem for a quadratic Diophantine equation
 involving a pair of integers $(k, r)$ with $\theta \approx \frac{k}{r}$.
 In the next section, we will use the pigeonhole principle to show that this equation does not have too many solutions, provided the Diophantine condition \eqref{1.4} holds.

The starting point is to apply the Poisson summation formula to flip the character summation in \eqref{eqn: 4 upper weight} from ``up to $x$" to ``up to $r/x$".
Let $k=\floor{r\theta}$.
Write $\theta = k/r + \theta'$ where $0\le \theta' < 1/r$.
This is a pragmatic rational approximation of $\theta$
that will prevent an increase in the ``conductor'' $r$ for the problem at hand.\footnote{This can be thought of as a special case of the ``conductor lowering'' trick in the delta method.
However, we believe it could be worthwhile to explore more complicated rational approximations of $\theta$, in order to sharpen our analysis.}
We define
\[f_{r,\chi}(n):=  \chi(n)e(\frac{kn}{r}),
\quad f_{\infty}(n) := w(\frac{n}{x}) e(n\theta'). \]
Then since $w(n/x) = 0$ for all integers $n\le 0$ and $n\ge x$,
we have
\begin{equation}
\label{poisson}
\sum_{1\le n\le x} \chi(n) e(n\theta) w(\frac{n}{x})
= \sum_{n\in \Z} f_{r,\chi}(n)f_{\infty}(n)
=  \sum_{m\in \Z} \hat{f}_{r,\chi}(\frac{m}{r}) \hat{f}_{\infty}(\frac{m}{r})
\end{equation}
by Poisson summation in $(\Z/r\Z) \times \R$,
where 
\[\hat{f}_{r,\chi}(\frac{m}{r}) =  \frac{1}{r}\sum_{t\in \Z/r\Z}\chi(t) e\left(\frac{(k+m)t}{r}\right) \]
and 
\[\hat{f}_{\infty}(\frac{m}{r}) = \int_{\R} w(\frac{t}{x}) e( (\theta'-\frac{m}{r})t  ) dt.   \]

We now estimate the Fourier coefficients $\hat{f}_{r,\chi}(\frac{m}{r})$
and $\hat{f}_{\infty}(\frac{m}{r})$.
For a fixed $m$, if $k+m \not\equiv 0 \bmod{r}$
then by applying standard properties of Gauss sums,
\begin{equation}
\label{gauss}
\hat{f}_{r,\chi}(\frac{m}{r}) = \chi(k+m)^{-1} \frac{C(\chi)}{r^{1/2}},
\end{equation}
where $\abs{C(\chi)} \le 1$ and $C(\chi)$ depends only on $\chi$.
Next, we turn to $\hat{f}_\infty$.

\begin{lemma}
[Fourier bounds]
\label{lem:fourier}
For all $A\ge 0$ we have
$$\hat{f}_{\infty}(\frac{m}{r}) \ll_A x \big(1+\frac{x\max(\abs{m}-1,0)}{r}\big)^{-A}.$$
\end{lemma}

\begin{proof}
Integrating by parts over $t$,
and recalling that $w$ is smooth on $\R$
and supported on $[0,1]$,
we get
$$\hat{f}_{\infty}(\frac{m}{r})
\ll_B \int_0^x x^{-B} \abs{\theta'-\frac{m}{r}}^{-B}\, dt$$
for all $B\ge 0$.
Moreover, $\abs{\theta'-\frac{m}{r}} \ge \max(\frac{\abs{m}-1}{r}, 0)$,
since $\abs{\theta'} \le \frac1r$.
Thus
$$\hat{f}_{\infty}(\frac{m}{r})
\ll_B x \big(x\max(\frac{\abs{m}-1}{r}, 0)\big)^{-B}.$$
Optimizing this bound over $B\in \{0,A\}$ gives the desired result.
\end{proof}

We first discuss the pesky terms in \eqref{poisson} with $m\equiv -k\bmod{r}$.
We have $$1+\frac{x\max(\abs{m}-1,0)}{r} \gg \frac{x\max(\abs{m},r)}{r}$$
for all $m\equiv -k\bmod{r}$,
because $k\sim \theta r$ and we may assume $0<\theta<1$.
Therefore,
\begin{equation}\label{eqn: m=k}
\sum_{m\equiv -k\bmod{r}} \hat{f}_{r,\chi}(\frac{m}{r}) \hat{f}_{\infty}(\frac{m}{r})
= \frac{\bm{1}_{\chi=\chi_0}}{r-1} \sum_{m\equiv -k\bmod{r}} \hat{f}_{\infty}(\frac{m}{r})
\ll \frac{x^{1-A}}{r-1}
\end{equation}
by Lemma~\ref{lem:fourier}, provided $A>1$.

Before expanding the fourth moment of \eqref{poisson}, we first reduce to the case where the variables all lie in the same range.
Write $\sum_{m\in \Z} = \sum_{j\ge 0} \sum_{m\in I_j}$, where $$I_0 = [-2-r/x, 2+r/x]$$
and $$I_j = \{\abs{m} \in (2^{j-1} (2+r/x), 2^j (2+r/x)]\}$$ for $j\ge 1$.
In view of Lemma~\ref{lem:fourier}, the ranges $j\ll 1$ will morally dominate.
Let $W_j = 2^{-j\delta}$ for some small $\delta>0$ to be chosen.
H\"{o}lder's inequality over $j$ gives the inequality
\begin{equation}
\begin{split}
\label{dyadic-holder}
&\abs{\sum_{m\in \Z} \hat{f}_{r,\chi}(\frac{m}{r}) \hat{f}_{\infty}(\frac{m}{r})}^4 \\
&\ll \abs{\sum_{m\equiv -k \bmod{r}} \hat{f}_{r,\chi}(\frac{m}{r}) \hat{f}_{\infty}(\frac{m}{r})}^4
+  \mathcal{W} \sum_{j\ge 0} W_j^{-3}
\abs{\sum_{\substack{m\in I_j \\ m\not\equiv -k \bmod{r}}} \frac{\chi(k+m)^{-1}}{r^{1/2}}
\hat{f}_{\infty}(\frac{m}{r})}^4
\end{split}
\end{equation}
by \eqref{gauss},
where $\mathcal{W} \defeq (\sum_{j\ge 0} W_j)^3 \ll_\delta 1$.
The first term will be negligible by \eqref{eqn: m=k}.


Let $T_j\defeq 2^j(2+r/x)$.
In particular, if $j\ge 1$, then $\abs{m}-1 \asymp T_j$ for all $m\in I_j$.
By orthogonality and Lemma~\ref{lem:fourier}, we have 
\begin{equation*}
\E_\chi
\abs{\sum_{\substack{m\in I_j \\ m\not\equiv -k \bmod{r}}} \frac{\chi(k+m)^{-1}}{r^{1/2}}
\hat{f}_{\infty}(\frac{m}{r})}^4
\ll
\left(\frac{x}{1 + (T_jx/r)^A \bm{1}_{j\ge 1}}\right)^4 \frac{1}{r^2} \mcal{N}_4(I_j),
\end{equation*}
where 
$\mcal{N}_4(I_j)$ is the number of solutions $(m_1,m_2,n_1,n_2)\in \{m\in I_j: m\not\equiv -k\bmod{r}\}^4$ to
$(k+m_1)(k+m_2) \equiv (k+n_1)(k+n_2) \bmod{r}$, or equivalently to
\begin{equation}\label{eqn: N4}
  k(m_1+m_2-n_1-n_2) \equiv n_1n_2-m_1m_2 \bmod r.  
\end{equation}
We write $S= m_1+m_2-n_1-n_2$ and $P= n_1n_2-m_1m_2$.
We have $kS \equiv P \bmod{r}$.

\begin{lemma}
[Injection]
\label{lem:injection}
For $(m_1,m_2,n_1,n_2)\in \Z^4$, let
$$\Phi(m_1,m_2,n_1,n_2)\defeq
(n_1-n_2+m_1-m_2,
n_1-n_2-m_1+m_2,
m_1+m_2) \in \Z^3.$$
Let $S,P\in \Z$.
Then $\Phi$ maps the set $\mcal{A}$ injectively into the set $\mcal{B}$, where
\begin{equation*}
\begin{split}
\mcal{A} &\defeq \{(m_1,m_2,n_1,n_2)\in \Z^4: m_1+m_2-n_1-n_2=S,\quad n_1n_2-m_1m_2=P\}, \\
\mcal{B} &\defeq \{(a,b,c)\in \Z^3: ab+2cS = S^2-4P\}.
\end{split}
\end{equation*}
\end{lemma}

\begin{proof}
Let $a,b,c$ be the linear forms in $m_1,m_2,n_1,n_2$
such that $\Phi(m_1,m_2,n_1,n_2) = (a,b,c)$.
Then the following polynomial identity holds:
$$(n_1-n_2)^2 - (m_1-m_2)^2 + c^2
= (c-S)^2 - 4P.$$
Therefore, $\Phi$ maps $\mcal{A}$ into $\mcal{B}$.
Moreover, this map is injective,
because the linear forms $a,b,c,S$ are linearly independent over $\Q$.
\end{proof}

\section{Point counting}

We would like to bound $\mcal{N}_4(I_j)$ defined in \eqref{eqn: N4}.
By Lemma~\ref{lem:injection}, we have
\begin{equation}
\mcal{N}_4(I_j)\le
\sum_{\substack{S\ll T_j,\; P\ll T_j^2 \\ kS \equiv P \bmod{r}}}
N_{S,P}(T_j)
\end{equation}
where
\begin{equation}\label{eqn: N(T)}
N_{S,P}(T) \defeq \#\{a,b,c\ll T: ab+2cS = S^2-4P\}.
\end{equation}
It will follow from Theorem~\ref{clean-counting}
that $\mcal{N}_4(I_j) \ll T_j^2 + \frac{T_j^4}{r}$.


We note that the equation $ab+2cS = S^2-4P$ implies that for $S\ll T$,
\begin{equation*}
ab+4P \equiv 0 \bmod{S},
\qquad ab+4P \ll TS+S^2 \ll TS.
\end{equation*}
Therefore,
\begin{equation*}
N_{S,P}(T)
\le \sum_{a,b\ll T} \bm{1}_{S\mid ab+4P} \bm{1}_{ab+4P \ll TS}.
\end{equation*}

\begin{lemma}
\label{simple-hyperbola}
Suppose $1\le u,v\le S\ll T$.
Then $$\sum_{\substack{a,b\ll T \\ (a,b)\equiv (u,v)\bmod{S}}} \bm{1}_{ab+4P \ll TS}
\ll \frac{T}{S} \log(2 + \frac{T}{S}).$$
\end{lemma}

\begin{proof}
There are $O(1)$ choices for $\abs{a}\le S$ with $a\equiv u\bmod{S}$.
Thus the total contribution from $\abs{a}\le S$ is
$$\ll \#\{b\ll T: b\equiv v\bmod{S}\} \ll T/S.$$
On the other hand, if $A\ge S$, then there are $O(A/S)$ choices for $\abs{a}\in (A,2A]$ with $a\equiv u\bmod{S}$,
so the total contribution from the dyadic interval $\abs{a}\in (A,2A]$ is
$$\ll \sum_{\substack{\abs{a}\in (A,2A] \\ a\equiv u\bmod{S}}} \#\{b\equiv v\bmod{S}: ab+4P\ll TS\}
\ll \sum_{\substack{\abs{a}\in (A,2A] \\ a\equiv u\bmod{S}}} (1 + \frac{TS}{\abs{aS}})
\ll \frac{A}{S} (1 + \frac{T}{A}).$$
On summing over $A\in \{1,2,4,8,\ldots\}$ with $S\le A\ll T$,
we get a total bound of $$\ll \frac{T}{S} \log(2 + \frac{T}{S}),$$
as desired.
\end{proof}
For any $S\ll T$ and $P\ll T^2$ with $S\ne 0$, we have
\begin{equation}
\label{simple-pass-to-congruence}
N_{S,P}(T)
\le \sum_{a,b\ll T} \bm{1}_{S\mid ab+4P} \bm{1}_{ab+4P \ll TS}
\ll \frac{T}{S} \log(2 + \frac{T}{S}) N(-4P,S),
\end{equation}
by Lemma~\ref{simple-hyperbola}, where
\begin{equation*}
N(d,q) \defeq \#\{(a,b)\in (\Z/q\Z)^2: ab \equiv d\bmod{q}\}.
\end{equation*}


\begin{lemma}
[Point counting]
\label{local-point-counting}
Let $d,q\in \Z$ with $q\ge 1$.
Then $N(d,q)\le \tau(\gcd(d,q)) q$, where $\tau(\cdot)$ is the divisor function. 
\end{lemma}

\begin{proof}
If $\gcd(q_1,q_2)=1$, then $N(d,q_1q_2) = N(d,q_1)N(d,q_2)$ by the Chinese remainder theorem.
Therefore, it suffices to prove the lemma when $q$ is a prime power.
Say $q = p^t$ and $\gcd(d,q) = p^m$.
Then clearly $t\ge m\ge 0$.
If $m=0$, then $N(d,q) = \phi(q) \le q$.
If $m=1$, then $N(d,q) = 2\phi(q) + \bm{1}_{t=1} \le 2q$.
If $m\ge 2$, then $N(d,q) = 2\phi(q) + p^2 N(d/p^2,q/p^2)$.
By induction on $m$, it follows that $N(d,q) \le (m+1)q$, as desired.
\end{proof}

We next count the number of solutions $(S, P)$
to the equation $kS \equiv P \bmod{r}$
by exploring the Diophantine approximation property of $\theta$.
\begin{lemma}
[Pigeonhole argument]
\label{pigeonhole}
Assume $\abs{q\theta-a} \gg \Upsilon(q)$ for all $(a,q)\in \Z\times \N$,
where $\Upsilon$ is a decreasing, nonnegative function.
For any integer $r$, let $k=\floor{r\theta}$.
Then for any $M,N\in \R$ with $$r/2>M\ge N\ge 1,$$ we have
\begin{equation}
\label{pigeon-goal}
\Upsilon{\left(\frac{N}{\#\{(S,P)\in [1,N]\times [-M,M]: kS \equiv P \bmod{r}\}}\right)}
\ll \frac{M}{r}.
\end{equation}
For example, if $\Upsilon(q) = \exp(-q^c)$, for some constant $c>0$, then
\begin{equation}
\label{pigeon-alg}
\left(\frac{N}{\#\{(S,P)\in [1,N]\times [-M,M]: kS \equiv P \bmod{r}\}}\right)^c
\gg \log{\left(\frac{r}{M}\right)} - O(1),
\end{equation}
so
\begin{equation}
\label{pigeon-flip}
\#\{S\ll N,\; P\ll M: kS \equiv P \bmod{r}\}
\ll \frac{N}{(\log(2+r/M))^{1/c}}.
\end{equation}
\end{lemma}

\begin{proof}
By the pigeonhole principle, there exists $(q,d)\in [1,N]\times [-2M,2M]$ such that
\begin{equation*}
kq\equiv d\bmod{r},
\qquad
q\le \frac{N}{\#\{(S,P)\in [1,N]\times [-M,M]: kS \equiv P \bmod{r}\}}.
\end{equation*}
For such a pair $(q,d)$, 
we have $kq = d+ra$ for some $a\in \Z$.
But by definition of $k$, we have $\abs{r\theta-k} < 1$.
Therefore,
\begin{equation*}
\abs{qr\theta-ra}
\le \abs{qr\theta-kq} + \abs{kq-ra}
< q + \abs{d}
\le N+2M
\le 3M,
\end{equation*}
whence $\abs{q\theta-a} \le 3M/r$.
Yet by assumption, $\abs{q\theta-a} \gg \Upsilon(q)$.
Since $\Upsilon(q)$ is decreasing, we immediately deduce \eqref{pigeon-goal}.
Now \eqref{pigeon-alg} follows from \eqref{pigeon-goal}.
Next, \eqref{pigeon-flip} follows from \eqref{pigeon-alg} if $r/M$ is sufficiently large.
On the other hand, \eqref{pigeon-flip} is trivial if $r\ll M$.
\end{proof}

\begin{theorem}
\label{clean-counting}
Assume $T\gg 1$ and let $N_{S, P}(T)$ be defined as in \eqref{eqn: N(T)}. 
Then
\begin{equation}
\sum_{\substack{S\ll T,\; P\ll T^2 \\ kS \equiv P \bmod{r}}}
N_{S,P}(T)
\ll T^2 + \frac{T^4}{r}.
\end{equation}
\end{theorem}

\begin{proof}
Since $N_{0,0}(T) \ll T^2$ and $N_{0,P}(T) \ll T\abs{P}^\ee$ if $P\ne 0$,
we have
\begin{equation*}
\sum_{\substack{P\ll T^2 \\ P \equiv 0 \bmod{r}}}
N_{0,P}(T)
\ll_\ee T^2 + \frac{T^2}{r} T^{1+2\ee}.
\end{equation*}
We may henceforth restrict attention to $S\ne 0$.
By \eqref{simple-pass-to-congruence} and Lemma~\ref{local-point-counting}, we have
\begin{equation*}
\begin{split}
\heartsuit
\defeq \sum_{\substack{S\ll T,\; P\ll T^2 \\ kS \equiv P \bmod{r}}}
N_{S,P}(T)
&\ll \sum_{\substack{S\ll T,\; P\ll T^2 \\ kS \equiv P \bmod{r}}}
\frac{T}{\abs{S}} \log(2 + \frac{T}{\abs{S}}) N(-4P,S) \\
&\ll \sum_{\substack{S\ll T,\; P\ll T^2 \\ kS \equiv P \bmod{r}}}
T \log(2 + \frac{T}{\abs{S}}) \tau(\gcd(P,S)),
\end{split}
\end{equation*}
where in the last step we use the sub-multiplicativity property $\tau(mn)\le \tau(m)\tau(n)$ with $m=4$.
Upon writing $(S,P) = (gS',gP')$ with $g = \gcd(S,P)\ge 1$,
and summing $\tau(g)$ over dyadic intervals $[G/2,G)$, we get
\begin{equation}
\begin{split}
\label{split-case}
\heartsuit
&\ll \sum_{\substack{G\in \{2,4,8,\ldots\} \\ G\ll T}}
\sum_{\substack{S\ll T/G,\; P\ll T^2/G \\ kS \equiv P \bmod{r}}}
T \log(2 + \frac{T}{\abs{GS}}) (G\log{G}) \\
&+\sum_{\substack{G\in \{2r,4r,8r,\ldots\} \\ G\ll T}}
\sum_{S\ll T/G,\; P\ll T^2/G}
T \log(2 + \frac{T}{\abs{GS}}) \tau(r) (\frac{G}{r}\log{\frac{G}{r}}),
\end{split}
\end{equation}
where the last line accounts for the contribution from $g\equiv 0\bmod{r}$.
The last line is
\begin{equation*}
\begin{split}
&\ll \sum_{\substack{G\in \{2r,4r,8r,\ldots\} \\ G\ll T}}
\sum_{S\ll T/G,\; P\ll T^2/G}
T (\frac{T}{\abs{GS}})^{0.1} \tau(r) (\frac{G}{r})^{1.1} \\
&\ll \sum_{\substack{G\in \{2r,4r,8r,\ldots\} \\ G\ll T}}
T (T/G) (T^2/G) \tau(r) (\frac{G}{r})^{1.1} \\
&\ll T (T/r) (T^2/r) \tau(r)
\ll \frac{T^4}{r}.
\end{split}
\end{equation*}

Let $\mcal{S}\in \{2,4,8,\ldots\}$ with $\mcal{S}\ll T/G$.
If $T^2/G\gg r$, then trivially
\begin{equation}
\#\{S\asymp \mcal{S},\; P\ll T^2/G: kS \equiv P \bmod{r}\}
\ll \mcal{S} \frac{T^2/G}{r},
\end{equation}
and otherwise, by Lemma~\ref{pigeonhole} with $c=1/3$ we have
\begin{equation}
\#\{S\asymp \mcal{S},\; P\ll T^2/G: kS \equiv P \bmod{r}\}
\ll \frac{\mcal{S}}{(\log(2+rG/T^2))^3}.
\end{equation}
Upon summing over all possible choices for $\mcal{S}$, we conclude that
the first line of \eqref{split-case} is
\begin{equation*}
\begin{split}
&\ll \sum_{\substack{G,\mcal{S}\in \{2,4,8,\ldots\} \\ \mcal{S}G\ll T}}
T\left(\frac{T}{\abs{G\mcal{S}}}\right)^{0.1} (G\log{G})
\mcal{S} \left(\frac{T^2}{rG}
+ \frac{\bm{1}_{G\gg T^2/r}}{(\log(2+rG/T^2))^3}\right) \\
&\ll \sum_{\substack{G\in \{2,4,8,\ldots\} \\ G\ll T}}
T^2 (\log{G})
\left(\frac{T^2}{rG}
+ \frac{\bm{1}_{G\gg T^2/r}}{(\log(2+rG/T^2))^3}\right) \\
&\ll \frac{T^4}{r} + T^2
\sum_{\substack{G\in \{2,4,8,\ldots\} \\ G\gg T^2/r}}
\frac{\log{G}}{(\log(2+rG/T^2))^3} \\
&\ll \frac{T^4}{r} + T^2 \log(2+\frac{T^2}{r}),
\end{split}
\end{equation*}
by Lemma~\ref{dyadic-G-numerics} below, applied with $G=2^j$ and $2^s\asymp T^2/r$.
But $\log(2+\frac{T^2}{r}) \ll 2+\frac{T^2}{r}$, so Theorem~\ref{clean-counting} follows.
\end{proof}

\begin{lemma}
\label{dyadic-G-numerics}
For any $s\in \Z$, we have
\begin{equation*}
\mcal{D}(s) \defeq
\sum_{j\ge \max(1,s)} \frac{j}{\max(1,j-s)^3}
\ll \max(1,s).
\end{equation*}
\end{lemma}

\begin{proof}
If $s\le 1$, then
\begin{equation*}
\mcal{D}(s)
\le \sum_{j\ge 1} \frac{j}{\max(1,j-1)^3}
\ll 1
= \max(1,s).
\end{equation*}
If $s\ge 1$, then a change of variables $j\mapsto j+s$ gives
\begin{equation*}
\mcal{D}(s)
= \sum_{j\ge 0} \frac{j+s}{\max(1,j)^3}
\ll 1+s
\ll \max(1,s).
\end{equation*}
Each case is satisfactory.
\end{proof}

Finally, we complete the proof of the lower bound of Theorem~\ref{THM: main} in the case $\sqrt{r}\le x\le r$.
By \eqref{dyadic-holder},
\eqref{eqn: m=k},
and Theorem~\ref{clean-counting}, we have
\begin{equation}
\mathcal{M}_4\defeq \E_\chi \abs{ \sum_{n\in \Z } \chi(n) e(n\theta) w(\frac{n}{x})}^4
\ll (x^{1-A})^4 + \frac{x^4}{r^2} \sum_{j\ge 0} \frac{2^{3j\delta}}
{1 + (T_jx/r)^{4A}\bm{1}_{j\ge 1}}
\left(T_j^2 + \frac{T_j^4}{r}\right),
\end{equation}
where $T_j=2^j(2+r/x)$.
Note that
\begin{equation*}
\sum_{j\ge 0} \frac{2^{3j\delta}}{1 + (T_jx/r)^{4A}\bm{1}_{j\ge 1}} T_j^2
\ll T_0^2 + \frac{1}{(T_1x/r)^{4A}} T_1^2
\ll (2+r/x)^2,
\end{equation*}
assuming $3\delta+2 < 4A$.
On the other hand,
\begin{equation*}
\sum_{j\ge 0} \frac{2^{3j\delta}}{1 + (T_jx/r)^{4A}\bm{1}_{j\ge 1}} \frac{T_j^4}{r}
\ll \frac{T_0^4}{r} + \frac{1}{(T_1x/r)^{4A}} \frac{T_1^4}{r}
\ll \frac{(2+r/x)^4}{r},
\end{equation*}
assuming $3\delta+4 < 4A$.
It follows that if $A > 1+3\delta/4$, then
\begin{equation}
\mathcal{M}_4
\ll \frac{x^4}{r^2} (2+r/x)^2
\left(1 + \frac{(2+r/x)^2}{r}\right)
\ll x^2(1+2x/r)^2,
\end{equation}
since $\frac{x}{r}(2+r/x) = 1+2x/r$ and $2+r/x \ll r^{1/2}$.
Since $x\le r$,
it follows that $\mathcal{M}_4 \ll x^2$.
This completes our analysis of the case $x\ge \sqrt{r}$ in \eqref{eqn: 4 upper weight}.
 
 \bibliographystyle{plain}
	\bibliography{character}{}
\end{document}